\documentclass[12pt]{amsart}

\usepackage{amscd,latexsym,amsthm,amsfonts,amssymb,amsmath,amsxtra}
\usepackage[mathscr]{eucal}
\usepackage{pst-plot}
\usepackage{hyperref}
\renewcommand\theequation{\thesection.\arabic{equation}}

\newcommand{\BC}{{\mathbb {C}}}

\newcommand{\BF}{{\mathbb {F}}}

\newcommand{\BQ}{{\mathbb {Q}}}

\newcommand{\BZ}{{\mathbb {Z}}}

\newcommand{\CH}{{\mathcal {H}}}

\newcommand{\CJ}{{\mathcal {J}}}

\newcommand{\CW}{{\mathcal {W}}}

\newcommand{\FY}{{\mathfrak {Y}}}

\newcommand{\Fm}{{\mathfrak {m}}}

\newcommand{\Fp}{{\mathfrak {p}}}

\newcommand{\ScO}{{\mathscr {O}}}

\newcommand{\ScW}{{\mathscr {W}}}

\newcommand{\ScZ}{{\mathscr {Z}}}

\newcommand{\End}{{\mathrm{End}}}

\newcommand{\GL}{{\mathrm{GL}}}

\newcommand{\Ind}{{\mathrm{Ind}}}

\newcommand{\Spec}{{\mathrm{Spec}}}
\newcommand{\SO}{{\mathrm{SO}}}

\newcommand{\Sp}{{\mathrm{Sp}}}

\newcommand{\Rep}{{\mathrm{Rep}}}
\DeclareMathOperator{\cInd}{c-Ind}

\newcommand{\zlb}{\overline{{\mathbb{Z}_{\ell}}}}
\newcommand{\qlb}{\overline{{\mathbb{Q}_{\ell}}}}

\newcommand{\wt}{\widetilde}

\newcommand{\bs}{\backslash}

\def\diag{{\rm diag}}

\newtheorem{thm}{Theorem}[section]
\newtheorem{cor}[thm]{Corollary}
\newtheorem{lem}[thm]{Lemma}
\newtheorem{prop}[thm]{Proposition}

\newtheorem {ques/conj}[thm]{Question/Conjecture}

\newtheorem{defn}[thm]{Definition}
\newtheorem{rmk}[thm]{Remark}

\makeatletter

\newcommand{\Rmnum}[1]{\expandafter\@slowromancap\romannumeral #1@}
\makeatother

\def\bJ{{\boldsymbol{\CJ}}}
\def\bN{{\boldsymbol{N}}}
\def\br{{\boldsymbol{r}}}

\begin{document}
\renewcommand{\theequation}{\arabic{equation}}
\numberwithin{equation}{section}

\title[Local Converse and Descent]{On the Local Converse Theorem and the Descent Theorem in families}

\author{Baiying Liu}
\address{Department of Mathematics\\
Purdue University\\
150 N. University Street\\
West Lafayette, IN 47907 USA}
\email{liu2053@purdue.edu}

\author{Gilbert Moss}
\address{Department of Mathematics\\
University of Utah\\
155 S 1400 E ROOM 233\\
Salt Lake City, UT 84112 USA}
\email{moss@math.utah.edu}

\begin{abstract}
In this paper, we prove an analogue of Jacquet's conjecture on the local converse theorem for $\ell$-adic families of co-Whittaker representations of $\GL_n(F)$, where $F$ is a finite extension of $\BQ_p$ and $\ell \neq p$. We also prove an analogue of Jacquet's conjecture for a descent theorem, which asks for the smallest collection of gamma factors determining the subring of definition of an $\ell$-adic family. These two theorems are closely related to the local Langlands correspondence in $\ell$-adic families.
\end{abstract}

\date{\today}
\thanks{The first named author is partially supported by NSF grant DMS-1702218 and by a start-up fund from the Department of Mathematics at Purdue University. The second named author was supported by the FSMP postdoctoral fellowship in France.}
\maketitle

\section{Introduction}\label{intro}

Let $F$ be a $p$-adic field whose residue field has order $q$, and let $G_n:=\GL_n(F)$. If $A$ is a commutative ring with unit, denote by $\Rep_A(G_n)$ the category of $A[G_n]$-modules which are smooth: the stabilizer of any element is open. Given irreducible generic representations $\pi_1$ and $\pi_2$ in $\Rep_{\BC}(G_n)$, Jacquet, Piatetski-Shapiro, and Shalika in \cite{JPSS83} defined gamma factors $\gamma(\pi_i\times\tau, s,\psi)$, $i=1,2$, for irreducible generic $\tau\in \Rep_{\BC}(G_t)$, a non-trivial additive character $\psi$ of $F$, and a complex variable $s$. If $\pi_1$ is isomorphic to $\pi_2$, then 
$$\gamma(\pi_1\times \tau,s,\psi) = \gamma(\pi_2\times \tau,s,\psi)\,,$$
 for all irreducible generic $\tau\in \Rep_{\BC}(G_t)$, for all $t\geq 1$. It is a natural problem to identify the smallest collection of representations $\tau$ such that the converse statement holds. In \cite{JL70} and \cite{JPSS79}, it was shown that when $n=2$ and $3$ respectively, the implication $$\gamma(\pi_1\times \tau,s,\psi) = \gamma(\pi_2\times \tau,s,\psi)\implies \pi_1\cong\pi_2$$ holds even when $\tau$ runs only over characters of $G_1$. In \cite{H93}, it was shown that for general $n$, the same implication holds when $\tau$ runs only over irreducible generic objects in $\Rep_{\BC}(G_t)$ for $t=1, 2, \ldots, n-1$.
 In \cite{Ch96, Ch06, CPS99, HO15}, the range of $t$ was improved to be $t=1, 2, \ldots, n-2$. 
 In general, the converse statement was conjectured by Jacquet to hold when $t$ varies from $1$ to $\lfloor \frac{n}{2}\rfloor$. This conjecture was recently proved by Chai in \cite{Ch16}, and by Jacquet and the first named author in \cite{JL16}, independently, using different methods. Jiang, Nien and Stevens (\cite{JNS15}) also proposed an approach towards Jacquet's conjecture based on the construction of supercuspidal representations in \cite{BK93}.
 Following this approach, a large part of Jacquet's conjecture is proved in \cite{JNS15}, and a combination of the results in \cite{JNS15, ALSX16} proves Jacquet's conjecture for $G_n$, $n$ prime.
 The analogue of Jacquet's conjecture for irreducible generic representations of $\GL_n$ over Archimedean fields is proved by 
 Adrian and Takeda in \cite{AT17}. 
 There is also an analogue of Jacquet's conjecture for irreducible generic representations of $\GL_n$ over finite fields proved by Nien in \cite{N14}. 
Local converse problems for groups other $\GL_n$ are studied in \cite{B95, B97} ($\mathrm{U}(2,1)$, $\Sp_4$), \cite{JS03} ($\SO_{2n+1}$), \cite{Z15, Z17a, Z17b, Z17c} ($\mathrm{U}(1,1)$, $\mathrm{U}(2,2)$, $\Sp_{2r}$, 
 $\mathrm{U}_{r,r}$, $\mathrm{U}_{2r+1}$).

Let $\ell\neq p$ be a prime number, let $k$ be an algebraically closed field of characteristic $\ell$, and let $W(k)$ be the ring of Witt vectors of $k$. That is, $W(k)$ is the smallest complete discrete valuation ring of characteristic zero whose residue field is $k$ (for instance, if $k=\overline{\BF_{\ell}}$ then $W(k)$ is isomorphic to the $\ell$-adic completion of the ring of integers in the maximal unramified extension of $\BQ_{\ell}$). Let $A$ be a Noetherian $W(k)$-algebra. An object in $\Rep_A(G_n)$ is an $\ell$-adic family of representations in the sense of algebraic geometry: given $\Fp\in \Spec(A)$ with residue field $\kappa(\Fp):=A_{\Fp}/\Fp A_{\Fp}$, the fiber $V\otimes_A\kappa(\Fp)$ gives a representation of $G_n$ on a $\kappa(\Fp)$-vector space. In this paper we follow the method in \cite{JL16} to prove two analogues of Jacquet's conjecture in the setting of $\ell$-adic families.

If $\pi$ is a simple $A[G_n]$-module, then for any ideal $I$ of $A$, $I\cdot\pi$ is either $0$ or all of $\pi$, so to have families that encode congruences, we do not use irreducible representations as our basic objects. Ihara's Lemma, and its conjectural generalization beyond $\GL_2$, imply that for representations arising in the cohomology of Shimura varieties, all irreducible subrepresentations of the contragredient are generic after taking the fiber at a maximal ideal of the global Hecke algebra. This motivates us to work with ``co-Whittaker" objects, that is, representations in $\Rep_A(G_n)$ that are generic with multiplicity one, admissible, and such that every nonzero quotient is generic (see Section~\ref{section:cowhittaker} for precise definitions). 

For any co-Whittaker representation $\pi$ in $\Rep_A(G_n)$, we define its Whittaker model $\CW(\pi,\psi)$ as the image of any nonzero homomorphism $\pi \to \Ind_{U_n}^{G_n}\psi_A$, where $\psi:F\to W(k)^{\times}$ is a nontrivial character extended to the group $U_n$ of upper triangular matrices, and $\psi_A = \psi \otimes_{W(k)} A$. Not being irreducible, there may be several non-isomorphic co-Whittaker representations with the same Whittaker model. Two co-Whittaker representations $\pi_1$, $\pi_2$ are equivalent (in the sense of Section~\ref{section:cowhittaker}) if and only if they have have the same Whittaker model -- this amounts to saying their ``supercuspidal supports'' are the same (see Section~\ref{section:cowhittaker} and Lemma \ref{equivalencecriterion}). In \cite{M16}, the second named author constructed gamma factors $\gamma(\pi\times\tau,X,\psi)$ in $(A\otimes B)((X))$, where $\pi\in \Rep_A(G_n)$ and $\tau\in\Rep_B(G_t)$ are co-Whittaker, $A$, $B$ are arbitrary Noetherian $W(k)$-algebras, and $X$ is a formal variable, see Section \ref{section:gamma} for more details. 
When $A =\mathbb{C}$, let $X = q^{-s+\frac{n-t}{2}}$, then $\gamma(\pi \times \tau, X, \psi)$ is exactly the gamma factor defined in \cite{JPSS83}.
The local converse theorem for $t=n-1$ is proven in \cite{M16}. Our first main result is proving an analogue of Jacquet's conjecture for co-Whittaker representations.

\begin{thm}\label{main1}
Let $A$ be a reduced, $\ell$-torsion free, finite-type $W(k)$-algebra and let $\pi_1$, $\pi_2$ be co-Whittaker $A[\GL_n(F)]$-modules with the same central character. If
$$\gamma(\pi_1\times \tau,X,\psi)=\gamma(\pi_2\times \tau,X,\psi)\,,$$ 
for all irreducible generic integral representations $\tau\in \Rep_{\qlb}(G_t)$ with $1\leq t\leq \lfloor\frac{n}{2}\rfloor$, then $\CW(\pi_1,\psi)=\CW(\pi_2,\psi)$ (equivalently, $\pi_1$ and $\pi_2$ have the same supercuspidal support; see Section~\ref{section:cowhittaker}).
\end{thm}

Recall that $\tau\in \Rep_{\qlb}(G_t)$ is \emph{integral} if it contains a stable $\zlb$-sublattice. Choose an isomorphism $\BC \cong \qlb$. Then in the special case of Theorem~\ref{main1} where $A=\BC$, we obtain a slightly stronger form of Jacquet's conjecture, where $\tau$ need only vary over irreducible generic representations that are integral. 

If $n\geq2$, we prove in Proposition~\ref{centralcharacter} that the condition on the central character is unnecessary, namely, the equalities of $\GL_1$-twist $\gamma$-factors imply $\pi_1$ and $\pi_2$ have the same central character. This is the analogue of \cite[Corollary 2.7]{JNS15}. 

Local converse theorems are especially useful in connection with the local Langlands correspondence. Converse theorems were used in \cite{H02} to show that equalities of twisted local factors uniquely characterize the local Langlands correspondence for $\GL_n$. Motivated by Ihara's Lemma and local-global compatibility results for $\GL_2$, Emerton and Helm conjectured in \cite{EH-families} a local Langlands correspondence for $\GL_n$ in $\ell$-adic families. Their conjecture assigns a co-Whittaker family $\pi(\rho)$ in $\Rep_A(G_n)$ to every $\ell$-adically continuous $A$-representation $\rho$ of the Weil group $W_F$. The family $\pi(\rho)$ ``interpolates local Langlands'' in the sense that, at generic points $\Fp$ of $\Spec(A)$, the fiber $\pi(\rho)\otimes_A\kappa(\Fp)$ corresponds to $\rho\otimes_A\kappa(\Fp)$ under a certain normalization of the classical local Langlands correspondence. They proved that, if $\pi(\rho)$ exists, it is the unique co-Whittaker family satisfying this interpolation property at the generic points. Recently, Helm and the second author (\cite{curtis, HM16}) proved the existence of $\pi(\rho)$, using local converse theorem and descent techniques in the co-Whittaker setting. Sharpening the techniques of \cite{M16} and \cite{HM16} to the level of Theorem~\ref{main1} will be useful in proving local Langlands correspondences in families beyond $\GL_n$.

We say a co-Whittaker $A[G_n]$-module \emph{descends} to a sub-$W(k)$-algebra $A'\subset A$ if there is a co-Whittaker $A'[G_n]$-module $\pi'$ such that $\pi$ is equivalent to $\pi'\otimes_{A'}A$. For $\pi\in \Rep_A(G_n)$ and $\tau\in \Rep_B(G_t)$ co-Whittaker, $\gamma(\pi\times\tau,X,\psi)$ defines an element of $(A\otimes B)((X))$, and if $\pi$ descends to $A'$, then $\gamma(\pi\times\tau,X,\psi)$ must have coefficients in the subring $A'\otimes B$. It is a natural problem to identify collections of representations $\tau$ over rings $B$ such that the converse statement holds. In \cite{HM16}, converse theorem techniques are used to prove a local gamma factor descent theorem with $t=n-1$, with $\tau$ being the compact induction $W_t:=\cInd_{U_t}^{G_t}\psi$, and with $B$ being $\ScZ_t$, the center of the category $\Rep_{W(k)}(G_{t})$. Our second main result sharpens this theorem to achieve another analogue of Jacquet's conjecture for descent.

\begin{thm}
\label{descent:intro}
Let $A$ be any Noetherian $W(k)$-algebra, let $A'\subset A$ be a sub-algebra, and suppose $\pi$ is a co-Whittaker $A[G_n]$-module whose central character is valued in $A'$. Assume $A$ is finitely generated as a module over $A'$. If $\gamma(\pi\times e'W_t,X,\psi)$ has coefficients in $A'\otimes e'\ScZ_t$ for all primitive idempotents $e'$ of $\ScZ_t$, and for $t=1,2\dots,\lfloor\frac{n}{2}\rfloor$, then $\pi$ descends to $A'$ (equivalently, the supercuspidal support of $\pi$ is valued in $A'$).
\end{thm}

If $n\geq 2$, we prove in Proposition~\ref{centralcharactervaluedA'} that the condition on the central character is automatically implied by the gamma factor condition for $t=1$. 

The theory of gamma factors for $\ell$-adically continuous families of representations of $W_F$ was developed in \cite{HM15}. We remark that, by applying the local Langlands correspondence in families, as formulated in \cite{HM16}, the converse and descent theorems stated here immediately give  analogous theorems for gamma factors of $\ell$-adically continuous families of representations of the Weil group $W_F$.
We briefly state these results as follows, referring the reader to \cite[\S 2]{HM15} for all notations and definitions. 
An $\ell$-inertial type $\nu$ is a finite dimensional representation of the prime-to-$\ell$ part of inertia. To any $n$-dimensional $\ell$-inertial type $\nu$, there is associated a primitive idempotent $e_{\nu}$ of $\ScZ_n$ via Vigneras' mod-$\ell$ semisimple local Langlands correspondence (\cite{V01}, \cite[Prop 10.1]{curtis}). There is a ring $R^{\nu}$ and a universal $\ell$-adically continuous representation $\rho_{\nu}:W_F\to \GL_n(R^{\nu})$ corresponding under local Langlands in families to $e'_{\nu}W_n$.

\begin{cor}
\label{corollary:parameter}
\begin{enumerate}
\item Let $A$ be a reduced, $\ell$-torsion free, finite-type $W(k)$-algebra, and let $\rho_1,\rho_2:W_F\to \GL_n(A)$ be $\ell$-adically continuous representations in the sense of \cite[\S2]{HM15} with the same determinant. If
$$\gamma(\rho_1\otimes \sigma,X,\psi)=\gamma(\rho_2\otimes \sigma,X,\psi)\,,$$ for all $\ell$-adically continuous representations $\sigma:W_F\to \GL_t(\zlb)$ with $1\leq t\leq \lfloor\frac{n}{2}\rfloor$, then $\rho_1$ and $\rho_2$ have the same semisimplification.
\item Let $A$ be any Noetherian $W(k)$-algebra, let $A'\subset A$ be a sub-algebra, and suppose $\rho:W_F\to \GL_n(A)$ is an $\ell$-adically continuous representation whose determinant character is valued in $A'$. Assume $A$ is finitely generated as a module over $A'$. If $\gamma(\rho\otimes \rho_{\nu'}, X, \psi)$ has coefficients in $A'\otimes R_{\nu'}$ for all $t$-dimensional $\ell$-inertial types $\nu'$, for $1\leq t\leq \lfloor\frac{n}{2}\rfloor$, then $\rho$ descends to $A'$, that is,
$\rho(W_F)$ lies in $\GL_n(A')$. 
\end{enumerate}
\end{cor}

In proving local converse theorems, there is a key vanishing result (Theorem~\ref{vanishingthm} in this paper) that is required to pass from known information on the Rankin--Selberg zeta integrals to desired information about the Whittaker function. This result was originally proven when $A=\BC$ in \cite{JS81} using harmonic analysis, and when $A$ is reduced, $\ell$-torsion free, and finite-type over $W(k)$ in \cite{M16}. The reduced and $\ell$-torsion free hypotheses were required in \cite{M16} to make use of certain algebro-geometric techniques, together with the theory of the integral Bernstein center, which is the reason why they appear in Theorem~\ref{main1}. We hope to remove these hypotheses in future work.

In Theorem~\ref{descent:intro}, the assumption that $A$ is finitely generated as a module over $A'$ is required to make use of \cite[Corollary 4.2]{HM16} (Proposition~\ref{rationalityfinite} in this paper). This technical result is needed because passing through the functional equation requires the gamma factor to be a rational function, whereas the technique of \cite{HM16}, which we exploit here, necessitates working with the gamma factor in its expansion as a power series. We hope to remove this complication in future work.

A natural question following Theorems \ref{main1}, \ref{descent:intro}, and Corollary \ref{corollary:parameter}, is whether the bound $\lfloor\frac{n}{2}\rfloor$ is sharp. 
The sharpness of $\lfloor\frac{n}{2}\rfloor$ for the local converse theorem when $A=\mathbb{C}$ is proved in \cite{ALST16} for the case of $n$ being prime and $p \geq \lfloor \frac{n}{2} \rfloor$. 
Thus, 
for $\ell$-adic families of co-Whittaker representations, we also expect the bound $\lfloor\frac{n}{2}\rfloor$ is sharp. We leave this discussion to future work. 

The structure of this paper is as follows. In section \ref{section:cowhittaker}, we introduce basic properties of co-Whittaker representations. In section \ref{section:gamma}, we briefly recall the theory of gamma factors in \cite{M16} and show that equality of $G_1$-twisted gamma factors implies equality of central characters. In Section 
\ref{section:vanishinglemmas}, we prove two basic lemmas, which play important roles in later sections. Theorem \ref{main1} will be proved in Section \ref{section:proofmain1}, with a proposition whose proof is deferred to  \ref{section:proofofprop3}. Theorem \ref{descent:intro} will be proved in Section \ref{section:descent}. 

\subsection*{Acknowledgements}

The authors would like to thank Matthew Emerton, David Helm, Herv\'{e} Jacquet, Dihua Jiang, and Shaun Stevens for their interest in this work, constant encouragement, and helpful conversations and suggestions.

\section{Co-Whittaker representations}
\label{section:cowhittaker}

Let $G=G_n=\GL_n(F)$. Recall that in $W(k)$, $p$ is invertible, there are all roots of unity of order prime to $\ell$, and we can fix an isomorphism $\overline{W(k)[\frac{1}{\ell}]}\cong\BC$, where $\overline{W(k)[\frac{1}{\ell}]}$ is an algebraic closure of the fraction field of $W(k)$. The base rings $A$ for our families will always have the structure of Noetherian $W(k)$-algebras.

This framework is natural when studying congruences mod $\ell$. For example, if $q\equiv 1$ mod $\ell$, there exist smooth characters $\chi_1, \chi_2:F^{\times}\to \zlb^{\times}$ such that $\chi_1$ is unramified but $\chi_2$ is ramified, and such that $\chi_1\equiv \chi_2$ mod $\Fm$, where $\Fm$ is the maximal ideal of $\zlb$. Let $A$ be the $W(k)$-algebra $\{(a,b)\in \zlb\times\zlb: a\equiv b \mod \Fm\}$. Then the congruence between $\chi_1$ and $\chi_2$ is captured by saying they interpolate in an $\ell$-adic family over $\Spec(A)= \{0\times \Fm, \Fm\times 0, \Fm\cdot A\}$. More precisely, the character $\chi: F^{\times}\to A^{\times}: a\mapsto (\chi_1(a),\chi_2(a))$ satisfies $\chi\otimes_A\kappa(0\times \Fm)\cong \chi_1$, $\chi\otimes_A\kappa(\Fm\times 0)\cong \chi_2$, and $\chi\otimes_A\kappa(\Fm)\cong \chi\mod\Fm$.

Let $B_n=T_nU_n$ be the standard Borel subgroup of $G_n$ consisting of upper triangular matrices with unipotent radical $U_n$ and $T_n$ the group of diagonal matrices. We fix a non-trivial additive character $\psi:F\to W(k)^{\times}$. Define a non-degenerate character $\psi_{U_n}$ on $U_n$ by 
$$\psi_{U_n}(u):=\psi\left(\sum_{i=1}^{n-1}u_{i,i+1}\right)\,,u\in U_n\,.$$
 We will drop the subscript and also refer to $\psi_{U_n}$ as $\psi$. If $A$ is a $W(k)$-algebra, let $\psi_A= \psi \otimes_{W(k)} A$. 

Define $V^{(n)}$ to be the $\psi$-coinvariants $V/V(U_n,\psi)$, where $V(U_n,\psi)$ is the submodule generated by $\{\psi(u)v-uv:u\in U_n,v\in V\}$. This functor is exact and, for any $A$-module $M$ there is a natural isomorphism 
$$(V\otimes_AM)^{(n)}\cong V^{(n)}\otimes_AM\,.$$

\begin{defn}
A smooth $A[G_n]$-module $V$ is co-Whittaker if the following conditions hold
\begin{enumerate}
\item $V$ is admissible as an $A[G_n]$-module,
\item $V^{(n)}$ is a free $A$-module of rank one,
\item if $Q$ is a quotient of $V$ such that $Q^{(n)}=0$, then $Q=0$.
\end{enumerate}
\end{defn}

For example, when $A=\BC$, $n=2$, and $B$ is the Borel subgroup, the normalized parabolic induction $i_B^G(\chi)$, where $\chi=\chi_1\otimes\chi_2$ is a character of $T$, is co-Whittaker. 

For another example, if $\chi_1$ and $\chi_2$ vary over unramified characters, this defines a geometric family over $\Spec(\BC[T_1^{\pm 1},T_2^{\pm 1}])$. More precisely, let $\chi_{univ}:T\to \BC[T_1^{\pm 1},T_2^{\pm 1}]^{\times}$ send $\diag(a,b)$ to $T_1^{v_F(a)}T_2^{v_F(b)}$, and consider the normalized induction $i_B^G(\chi_{univ})$. On points $\Fp=(T_1-x_1,T_2-x_2)$, the fiber $i_B^G(\chi_{univ})\otimes\kappa(\Fp)$ is irreducible on the open subset of points where $x_1x_2^{-1}\neq q^{\pm 1}$. At points where $x_1x_2^{-1}= q^{\pm 1}$ the fiber $i_B^G(\chi_{univ})\otimes\kappa(\Fp)$ is reducible, but has a unique irreducible generic quotient, which is a twist of the Steinberg representation by an unramified character of $F^{\times}$. It follows that $i_B^G(\chi_{univ})\otimes\kappa(\Fp)$ is co-Whittaker as a $\BC[T_1^{\pm 1},T_2^{\pm 1}][G_2]$-module.

If $V$ and $V'$ are co-Whittaker, any nonzero $G$-equivariant map $V\to V'$ is surjective, as otherwise the cokernel would be a nongeneric quotient. In this case $V$ is said to dominate $V'$. We say $V$ and $V'$ are {\it equivalent} if there exists a co-Whittaker $A[G_n]$-module $V''$ dominating both $V$ and $V'$. This is an equivalence relation on isomorphism classes of co-Whittaker modules.

In \cite{h_whitt}, Helm constructs a co-Whittaker module which is ``universal'' up to this notion of equivalence. The key tool is the integral Bernstein center of $G_n$, i.e. the center of the category $\Rep_{W(k)}(G_n)$.

The center of an abelian category is the endomorphism ring of the identity functor, in other words the ring of natural transformations from the identity functor to itself. It acts on every object in the category in a way compatible with all morphisms. We denote by $\ScZ_n$ the center of $\Rep_{W(k)}(G_n)$. 

For any co-Whittaker $A[G_n]$-module $V$, the map $A\to \End_{A[G_n]}(V)$ is an isomorphism (c.f. \cite[Prop 6.2]{h_whitt}), and thus there exists a map $f_V:\ScZ_n \to A$, which we call the \emph{supercuspidal support} of $V$. Note that $V$ also admits a central character $\omega_V:F^{\times}\to A^{\times}$.

A primitive idempotent $e$ of $\ScZ_n$ gives rise to a direct factor category $e\Rep_{W(k)}(G_n)$, which is the full subcategory of $\Rep_{W(k)}(G_n)$ on which $e$ acts as the identity. As described in \cite{h_bern}, the primitive idempotents in $\ScZ_n$ are in bijection with inertial equivalence classes of pairs $(L,\pi)$, where $L$ is a Levi subgroup of $G_n$ and $\pi$ is an irreducible supercuspidal $k$-representation of $L$. If $e$ is the idempotent corresponding to the pair $(L,\pi)$, then a representation $V$ in $\Rep_{W(k)}(G_n)$ lies in $e\Rep_{W(k)}(G_n)$ if and only if every simple subquotient of $V$ has mod-$\ell$ inertial supercuspidal support given by $(L,\pi)$ in the sense of \cite[Def 4.12]{h_bern}.

\begin{thm}[\cite{h_bern}, Thm 10.8]
Let $e$ be any primitive idempotent of $\ScZ_n$.
The ring $e\ScZ_n$ is a finitely generated, reduced, $\ell$-torsion free $W(k)$-algebra.
\end{thm}

Now let $A$ be a $W(k)$-algebra, and let $V$ be a co-Whittaker $A[G_n]$-module.  Suppose further that $V$ lies in $e\Rep_{W(k)}(G_n)$ for some primitive idempotent $e$ (so the supercuspidal support map $f_V$ factors through the projection $\ScZ_n\to e\ScZ_n$). Let $W_n$ be the smooth $W(k)[G_n]$-module $\cInd_{U_n}^{G_n} \psi$.  For any
primitive idempotent $e$ of $\ScZ_n$, we have an action of $e\ScZ_n$ on $eW_n$.

\begin{thm}[\cite{h_whitt}, Theorem 6.3]
\label{universalcowhitt} 
Let $e$ be any primitive idempotent of $\ScZ_n$.
The smooth $e\ScZ_n[G_n]$-module $eW_n$
is a co-Whittaker $e\ScZ_n[G_n]$-module. If $A$ is Noetherian and has an $e\ScZ_n$-algebra structure, the module $eW_n \otimes_{e\ScZ_n} A$
is a co-Whittaker $A[G_n]$-module. Conversely, 
$V$ is dominated by $eW_n \otimes_{e\ScZ_n, f_V} A$.
\end{thm}

We thus say that, up to the equivalence relation induced by dominance, $eW_n$ is
the universal co-Whittaker module in $e\Rep_{W(k)}(G_n)$.

By definition a co-Whittaker module admits a nonzero Whittaker functional $\iota: V\to \Ind_{U_n}^{G_n}\psi_A$ whose image, denoted by $\ScW(V,\psi)$ and called the ($A$-valued) \emph{Whittaker model of $V$} with respect to $\psi$, is independent of the choice of Whittaker functional. 
Fix a nonzero Whittaker functional $\iota$, for $v \in V$, let 
$$W_v=\iota(v) \in \ScW(V,\psi)\,.$$

\begin{lem}
\label{equivalencecriterion}
Let $A$ be a Noetherian ring and suppose $V_1$ and $V_2$ are two co-Whittaker $A[G_n]$-modules. The following are equivalent:
\begin{enumerate}
\item There exists $W$ in $\ScW(V_1,\psi)\cap \ScW(V_2,\psi)$ such that $W(g)\in A^{\times}$ for some $g\in G$.
\item $\ScW(V_1,\psi)= \ScW(V_2,\psi)$.
\item $f_{V_1}= f_{V_2}$.
\item $V_1$ and $V_2$ are equivalent.
\end{enumerate}
\end{lem}
\begin{proof}
First the equivalence of (3) and (4). Any co-Whittaker $A[G_n]$-module $V$ is a direct sum of the subrepresentations $e_iV\in e_i\Rep_{W(k)}(G_n)$; this sum is finite since $A$ is Noetherian. The equivalence then follows from Theorem \ref{universalcowhitt}. Next we show (1)$\implies$(2). Let $I=\ScW(V_1,\psi)\cap \ScW(V_2,\psi)$. The assumption that $W$ takes values in $A^{\times}$ guarantees that $I\otimes\kappa(\Fp)$ is nonzero for all $\Fp\in \Spec(A)$, and $(I\otimes\kappa(\Fp))^{(n)}=(\ScW(V_1,\psi)\otimes\kappa(\Fp))^{(n)}$. The finitely generated $A$-module $(\ScW(V_1,\psi)/I)^{(n)}\otimes\kappa(\Fp)$ is therefore zero for all $\Fp$, and by Nakayama's lemma the localizations $(\ScW(V_1,\psi)/I)^{(n)}_{\Fp}$ are zero for all $\Fp$, hence $(\ScW(V_1,\psi)/I)^{(n)}=0$. Since $\ScW(V_1,\psi)$ is co-Whittaker, we get $I=\ScW(V_1,\psi)$, and the parallel argument gives the same for $V_2$. For (2)$\implies$(1) we note that upon restriction to $P_n$, $\ScW(V_i,\psi)$ contains $\cInd_{U_n}^{P_n}\psi_A$, which has functions valued in $A^{\times}$. The equivalence of $(3)$ and $(2)$ were proven in \cite[Prop 6.2]{M16}.
\end{proof}

There is a duality operation on co-Whittaker modules which interpolates the contragredient across a co-Whittaker family (\cite[Prop 2.6]{HM16}). If $V$ is a smooth $W(k)[G_n]$-module, let $V^{\iota}$ denote the $W(k)[G_n]$-module with the same underlying $W(k)$-module structure, and for which the $G_n$ action, which we will denote by $g\cdot v$, is given by $g\cdot v = g^{\iota}v$, where $g^{\iota}={ ^tg}^{-1}$. This duality has a very concrete interpretation in terms of Whittaker functions. Let 
$$\omega_{n,m} = \begin{pmatrix}
I_{n-m}& 0 \\
0 & \omega_m
\end{pmatrix},$$
 where $\omega_m$ is the longest Weyl element of $G_m$. For any function $W$ on $G_n$, let $\wt{W}(g)=W(\omega_n g^{\iota})$. If $W$ is in $\ScW(V,\psi)$, then $\wt{W}$ is in $\ScW(V^{\iota},\psi^{-1})$.

\section{Rankin-Selberg theory and gamma factors}
\label{section:gamma}

Let $A$ and $B$ be Noetherian $W(k)$-algebras and let $R = A\otimes_{W(k)}B$. Let $V$ and $V'$ be $A[G_n]$- and $B[G_m]$-modules respectively, where $m<n$, and suppose both $V$ and $V'$ are of Whittaker type. For $W\in \CW(V,\psi)$ and $W'\in \CW(V', \psi^{-1})$, 
and for $0\leq j \leq n-m-1$, we define the formal series with coefficients in $R$ ($X$ is a formal variable): 
\begin{align*}
& \ \Psi(W,W',X;j)\\
:= & \ \sum_{r\in \BZ}\int_{M_{j,m}(F)}\int_{N_m\backslash\{g\in G_m:v(\det g)=r\}}                    \left(W\left(\begin{smallmatrix}g &&\\x&I_j &\\&&I_{n-m-j} \end{smallmatrix}\right)\otimes W'(g)\right)X^rdg dx\,.
\end{align*}

Let $S$ be the multiplicative system in $R[X, X^{-1}]$ consists of all polynomials whose leading and trailing coefficients are units.

\begin{thm}[Moss \cite{M16}]
\label{fnleqn}
\begin{enumerate}
\item Both $\Psi(W,W',X)$ and $\Psi(W,W',X;j)$ are elements in $S^{-1}(R[X, X^{-1}])$. 

\item There exists a unique element $\gamma (V \times V', X, \psi) \in S^{-1}(R[X, X^{-1}])$ such that 
\begin{align*}
& \ \Psi(W, W', X; j) \gamma(V \times V', X, \psi) \omega_{V'}(-1)^{n-1} 
\\
= & \ \Psi(\omega_{n,m}\wt{W}, \wt{W'}, \frac{q^{n-m-1}}{X}; n-m-1-j)\,,
\end{align*}
for any $W \in \CW(V,\psi)$, $W' \in \CW(V', \psi)$ and for any $0 \leq j \leq n-m-1$. 
\end{enumerate}
\end{thm}

\begin{prop}\label{prop:central}
Let $A$ be a Noetherian $W(k)$-algebra, let $\pi$ be a co-Whittaker $A[G_n]$-module with $n\geq 2$. Then there exists an integer $m_{\pi}$ such that for any character $\chi:F^{\times}\to W(k)^{\times}$ of conductor $m\geq m_{\pi}$, and any $c\in \Fp^{-m}$ satisfying $\chi(1+x)=\psi(cx)$ for any 
$x\in \Fp^{[m/2]+1}$, we have
$$\gamma(\pi\times \chi,X,\psi)=\omega_{\pi}(c)^{-1}\gamma(1_A\times \chi, X,\psi)^n\,,$$ where $1_A$ denotes the trivial character of $F^{\times}$.
\end{prop}
\begin{proof}
To begin, assume that $A$ is reduced and $\ell$-torsion free. There are finitely many minimal prime ideals $\Fp_1,\dots,\Fp_r$ of $A$. If $\kappa(\Fp_i)$ denotes the residue field $\text{Frac}(A/\Fp_i)$, then $\kappa(\Fp_i)$ has characteristic zero and contains $W(k)$. As $\pi$ is defined over a subalgebra which is finite type over $W(k)$, we may assume without loss of generality that $A$ is finite-type over $W(k)$ and choose an isomorphism $\overline{\kappa(\Fp_i)}\cong \BC$. Let $\pi_{i,0}$ denote the cosocle of $\pi\otimes_A\kappa(\Fp_i)$, which is absolutely irreducible generic by definition. 

In \cite[Proposition 2.6]{JNS15}, they prove the analogous proposition for irreducible generic representations of $G_n$ over $\BC$. Therefore, for all $i$, there exists $m_{\pi_{i,0}}$ such that
$$\gamma(\pi_{i,0}\times\chi,X,\psi)=\omega_{\pi_{i,0}}(c)^{-1}\gamma(1_{\overline{\kappa(\Fp_i)}}\times\chi,X,\psi)^n\,,$$ 
for any character $\chi:F^{\times}\to W(k)^{\times}$ of conductor $m\geq m_{\pi_{i,0}}$, and any $c\in \Fp^{-m}$ satisfying $\chi(1+x)=\psi(cx)$ for any 
$x\in \Fp^{[m/2]+1}$. On the other hand, $$\gamma(\pi,X,\psi)\equiv\gamma(\pi_{i,0}\times\chi,X,\psi) \mod\Fp_i\,,$$ 
by the compatibility of the gamma factor with homomorphisms of the base ring. For the same reason, we also have
$$\omega_{\pi}(c)^{-1}\gamma(1_A\times\chi,X,\psi)^n\equiv\omega_{\pi_{i,0}}(c)^{-1}\gamma(1_{\overline{\kappa(\Fp_i)}}\times\chi,X,\psi)^n\mod \Fp_i\,.$$ 
Therefore, for $m$ larger than the maximum of all $m_{\pi_{i,0}}$, the difference
$$\gamma(\pi,X,\psi)-\omega_{\pi}(c)^{-1}\gamma(1_A\times\chi,X,\psi)^n$$ lies in $\Fp_i$ for all $i$. Since $A$ is reduced, $\bigcap_i\Fp_i=0$, hence this difference is zero.

To go beyond the case where $A$ is reduced and $\ell$-torsion free we recall that each component $e\ScZ_n$ of the integral Bernstein center is reduced and $\ell$-torsion free. Without loss of generality, we may assume that $e\pi=\pi$ for a primitive idempotent $e$ of $\ScZ_n$. By Theorem~\ref{universalcowhitt}, $\pi$ is the base change, from $e\ScZ_n$ to $A$, of the co-Whittaker $e\ScZ_n[G_n]$-module $W_n$. We now apply the preceding paragraph to the co-Whittaker $e\ScZ_n[G_n]$-module $W_n$ to conclude that
$$\gamma(W_n\times\chi,X,\psi) = \omega_{W_n}(c)^{-1}\gamma(1_A\times\chi,X,\psi)^n\,,$$ 
where $\omega_{W_n}:F^{\times}\to e\ScZ_n$ denotes the central character of $W_n$. In general, there is a homomorphism $f_{\pi}:e\ScZ_n\to A$ such that 
$$\gamma(\pi\times\chi,X,\psi) = f_{\pi}(\gamma(W_n\times\chi,X,\psi)\,,$$ 
and similarly for $\omega_\pi$ and $\gamma(1_A\times\chi,X,\psi)^n$. The equal quantities over $e\ScZ_n$ map to equal quantities over $A$, which proves the theorem.
\end{proof}

The following result is an analogue of \cite[Corollary 2.7]{JNS15}.

\begin{prop}\label{centralcharacter}
Let $A$ be a Noetherian $W(k)$-algebra, let $\pi_1, \pi_2$ be co-Whittaker $A[G_n]$-modules with $n\geq 2$. Assume that 
$$\gamma(\pi_1\times \chi,X,\psi)=\gamma(\pi_2 \times \chi,X,\psi)\,,$$
for any character $\chi:F^{\times}\to W(k)^{\times}$. 
Then $\omega_{\pi_1} = \omega_{\pi_2}$. 
\end{prop}

\begin{proof}
Let $m_{\pi_{1,i,0}}, m_{\pi_{2,i,0}}$ be the numbers given in the proof of Proposition \ref{prop:central} for $\pi_1$ and $\pi_2$, respectively, $i=1,\dots,r$. Let $m_0$ the max of $\{m_{\pi_{1,i,0}}, m_{\pi_{2,i,0}}, i=1,\dots,r\}$. 

For any $c \in \Fp^{-m} \bs \Fp^{1-m}$, with $m \geq m_0$, there exists a character $\chi_c$ of conductor $m$ such that $\chi_c(1 + x) = \psi(cx)$ for $x \in \Fp^{[m/2]+1}$; thus Proposition \ref{prop:central} implies
$\omega_{\pi_1}(c)=\omega_{\pi_2}(c)$. 
Since any element of $F^{\times}$ can be expressed as the quotient of two elements of valuation at most $-m$, we deduce that $\omega_{\pi_1}=\omega_{\pi_2}$.
\end{proof}

\section{Two lemmas}
\label{section:vanishinglemmas}

In this section, we prove two lemmas, which will play important roles in later sections. 

For two $W(k)$-algebras $A$, $B$, and $\phi_1\in \cInd_{U_n}^{G_n}\psi_A$, $\phi_2\in \Ind_{U_n}^{G_n}\psi_B^{-1}$, we denote by $\langle \phi_1,\phi_2\rangle$ the element
$$\int_{U_n\backslash G_n}\phi_1(x)\otimes \phi_2(x) dx\in A\otimes_{W(k)}B\,.$$ 
Under certain hypotheses on $A$, we can detect the vanishing of $\phi_1$ by letting $\phi_2$ run over the collection of all Whittaker functions valued in $B=\zlb$.
We recall the following theorem in \cite{M16}. 

\begin{thm}[\cite{M16} Thm 6.4]
\label{vanishingthm}
Suppose $A$ is a finite-type, reduced, $\ell$-torsion free $W(k)$-algebra. Suppose $H\neq 0$ is an element of $\cInd_{U_n}^{G_n}\psi_A$. Then there exists an irreducible generic integral $\qlb$-representation $V'$ with $\zlb$-integral structure $V$, such that there is a Whittaker function $W\in \ScW(V^{\iota},\psi_{\zlb}^{-1})$ satisfying $\langle H,W\rangle \neq 0$ in $A\otimes_{W(k)}\zlb$.
\end{thm}

The following lemma is an analogue of 
\cite[Lemma 2.4]{JL16}. 

\begin{lem}
\label{vanishinglem}
Let $A$ be as in Thm \ref{vanishingthm} and let $\pi_1$ and $\pi_2$ be two co-Whittaker $A[G_n]$-modules. Suppose $t\leq n-2$ and $j$ with $0\leq j\leq t$. Suppose that $W^1$ and $W^2$ are elements in the Whittaker models of $\pi_1$ and $\pi_2$, respectively. Suppose further that for all irreducible generic integral representations $\tau$ in $\Rep_{\qlb}(G_{n-t-1})$ we have
$$\Psi(X,W^1,W';j)=\Psi(X,W^2,W';j)$$ for all $W'\in\ScW(\tau,\psi_{\qlb}^{-1})$. Then
$$\int W^1\begin{pmatrix}
I_{n-t-1} & 0 & 0 \\
X & I_j & 0\\
0 & 0 & I_{t+1-j}
\end{pmatrix} dX = \int W^2\begin{pmatrix}
I_{n-t-1} & 0 & 0 \\
X & I_j & 0\\
0 & 0 & I_{t+1-j}
\end{pmatrix} dX\,,$$
where the integrals are over $X \in M_{j \times (n-t-1)}(F)$. 
\end{lem}

\begin{proof}
For $j=0$, the assumption is that 
\begin{align}
\begin{split}
\ & \sum_r\int_{(U_{n-t-1} \bs G_{n-t-1})^r} W^1\begin{pmatrix}
g & 0 \\
0 & I_{t+1}
\end{pmatrix}\otimes W'(g) X^r dg\\
= \ & \sum_r\int_{(U_{n-t-1} \bs G_{n-t-1})^r} W^2\begin{pmatrix}
g & 0 \\
0 & I_{t+1}
\end{pmatrix} \otimes W'(g) X^r dg\,,
\end{split}
\end{align}
for all $W'$, where $(U_{n-t-1} \bs G_{n-t-1})^r$ is the subset of $U_{n-t-1} \bs G_{n-t-1}$ consisting of elements whose determinant has valuation $r$. The conclusion is that $W^1(I_n)=W^2(I_n)$. 
Indeed, recall that given $r>0$ the relations
\[ \lvert \det g\rvert=r\,,\, W^i  \left(\begin{array}{cc} g & 0 \\ 0 & I_{t+1} \end{array}\right) \neq 0\]
imply that $g$ is in a set compact modulo $U_{n-t-1}$. Taking $H_r$ to be $$\left(W^1 \left(\begin{array}{cc} g & 0 \\ 0 & I_{t+1} \end{array}\right)-W^2 \left(\begin{array}{cc} g & 0 \\ 0 & I_{t+1} \end{array}\right)\right)\Phi_r(g)\,,$$ where $\Phi_r$ is the characteristic function of $(U_{n-t-1} \bs G_{n-t-1})^r$, we have, by assumption, that $\langle H_r,W'\rangle=0$ for all Whittaker functions $W'$ of irreducible generic integral representations $\tau$, and hence of the contragredient representations $\tau^{\iota}$. On the other hand, by Theorem \ref{vanishingthm}, if $H_r$ were nonzero, there would be some $V'$, $V$ as in Theorem \ref{vanishingthm} such that $\langle H_r,W'\rangle\neq 0$ in $A\otimes_{W(k)}\zlb$ for some $W'\in \ScW(V^{\iota},\psi_{\zlb}^{-1})$. Then taking $W'\otimes 1$ in $\ScW((V')^{\iota},\psi_{\qlb}^{-1})$, we have $\langle H_r,W'\otimes 1\rangle=\langle H_r,W'\rangle\otimes 1\neq 0$, a contradiction, so $H_r$ must equal zero. Since $H_r=0$ for all $r$, we then have $W^1-W^2=0$ on $G_{n-t-1}$.

For $0 < j \leq t$, one observes that there is a compact subset $\Omega$ of $M_{j \times (n-t-1)}(F)$ such that for all $g \in G_{n-t-1}$ and $i=1,2$, 
$$W^i \begin{pmatrix}
g & 0 & 0 \\
X & I_j & 0\\
0 & 0 & I_{t+1-j}
\end{pmatrix}
\neq 0$$
implies that $X \in \Omega$. 
Thus, for $i=1,2$, there is an element $W^i_0 \in \CW(\pi_i, \psi)$ such that for all $g \in G_{n-t-1}$
$$\int_{M_{j \times (n-t-1)}(F)} W^i \begin{pmatrix}
g & 0 & 0 \\
X & I_j & 0\\
0 & 0 & I_{t+1-j}
\end{pmatrix} dX= W^i_0 \begin{pmatrix}
g & 0 & 0 \\
0 & I_j & 0\\
0 & 0 & I_{t+1-j}
\end{pmatrix}.$$
We are therefore reduced to the case $j=0$. 
\end{proof}

\begin{rmk}
There is a gap in the proof of \cite[Theorem 1.1]{M16}. \emph{A priori}, infinitely many different extensions $\ScO$ of $W(k)$ could be necessary as $W_i$ and $m$ vary in \cite[\S 6.1]{M16}. Thus, the argument in that paper only proves the following slightly weaker result: ``if $\gamma(V_1\times V',X,\psi)=\gamma(V_2\times V',X,\psi)$ for all irreducible generic integral representations $V'$ of $G_{n-1}$ over $\qlb$, then $V_1$ and $V_2$ have the same supercuspidal support.'' The authors believe \cite[Theorem 1.1]{M16} is correct as stated; this will be addressed in future work.
\end{rmk}

At the cost of taking larger rings $B=e'\ScZ_n$, for primitive idempotents $e'$ of $\ScZ_n$, but without any hypotheses on $A$, the collection of Whittaker functions over $B$ can detect a subring in which $\phi_1$ takes values, essentially by duality.
We recall the following result in \cite{HM16}. 

\begin{thm}[Cor 3.6, \cite{HM16}]
\label{dualityprop}
Let $A'$ be a $W(k)$-subalgebra of $A$, and suppose $H \neq 0$ is an element of $\cInd_{U_n}^{G_n}\psi_A$. If $\langle H, W'\rangle$ lies in $A'\otimes e'\ScZ_n$ for all primitive idempotents $e'$ of $\ScZ_n$ and all $W' \in \CW(e'W_n,\psi^{-1})$, then $H$ lies in $\cInd_{U_n}^{G_n}\psi_{A'}$.
\end{thm}

\begin{lem}
\label{vanishinglemdescent}
Let $A$ be any Noetherian $W(k)$-algebra and $A'\subset A$ a sub-algebra, and let $\pi$ be a co-Whittaker $A[G_n]$-module. Suppose $t\leq n-2$ and $0\leq j\leq t$. Suppose that for $W$ in $\CW(\pi,\psi)$, the power series $\Psi(W,W',X;j)$ takes coefficients in $A'$ for all primitive idempotents $e'$ of $\ScZ_{n-t-1}$, and all $W' \in \CW(e'W_{n-t-1},\psi^{-1})$. Then $$\int W\begin{pmatrix}
I_{n-t-1} & 0 & 0 \\
X & I_j & 0\\
0 & 0 & I_{t+1-j}
\end{pmatrix} dX \text{ is in }A'\,,$$
where the integrals are over $X \in M_{j \times (n-t-1)}(F)$. 
\end{lem}
\begin{proof}
For $j=0$, the assumption is that 
\begin{align}
\ & \sum_r\int_{(U_{n-t-1} \bs G_{n-t-1})^r} W\begin{pmatrix}
g & 0 \\
0 & I_{t+1}
\end{pmatrix}\otimes W'(g) X^r dg
\end{align}
is in $A'[[X]][X^{-1}]$ for all $W'$, where $(U_{n-t-1} \bs G_{n-t-1})^r$ is the subset of $U_{n-t-1} \bs G_{n-t-1}$ consisting of elements whose determinant has valuation $r$. The conclusion is that $W(I_n)$ is in $A'$. Indeed, recall that given $r>0$ the relations
\[ \lvert \det g\rvert=r\,,\, W  \left(\begin{array}{cc} g & 0 \\ 0 & I_{t+1} \end{array}\right) \neq 0\]
imply that $g$ is in a set compact modulo $U_{n-t-1}$. Taking $$H_r = W \left(\begin{smallmatrix} g & 0 \\ 0 & I_{t+1} \end{smallmatrix}\right)\Phi_r(g)\,,$$
where $\Phi_r$ is the characteristic function of $(U_{n-t-1} \bs G_{n-t-1})^r$, we have, by assumption, that $\langle H_r,W'\rangle\in A'\otimes e'\ScZ_{n-t-1}$ for all Whittaker functions $W'$ of $e'W_{n-t-1}$, and for all $e'$. By Proposition \ref{dualityprop} $H_r$ must take values in $A'$. Since this is true for all $r$, we then have $W \left(\begin{smallmatrix} g & 0 \\ 0 & I_{t+1} \end{smallmatrix}\right)$ is in $A'$.

For $0 < j \leq t$, one observes that there is a compact subset $\Omega$ of $M_{j \times (n-t-1)}(F)$ such that for all $g \in G_{n-t-1}$, 
$$W\begin{pmatrix}
g & 0 & 0 \\
X & I_j & 0\\
0 & 0 & I_{t+1-j}
\end{pmatrix}
\neq 0$$
implies that $X \in \Omega$. 
Thus, there is an element $W_0 \in \CW(\pi, \psi)$ such that for all $g \in G_{n-t-1}$
$$\int_{M_{j \times (n-t-1)}(F)} W \begin{pmatrix}
g & 0 & 0 \\
X & I_j & 0\\
0 & 0 & I_{t+1-j}
\end{pmatrix} dX= W_0 \begin{pmatrix}
g & 0 & 0 \\
0 & I_j & 0\\
0 & 0 & I_{t+1-j}
\end{pmatrix}.$$
We are therefore reduced to the case $j=0$. 
\end{proof}

\section{Proof of Theorem \ref{main1}}
\label{section:proofmain1}

In this section, we prove Theorem \ref{main1}, following the proof of the \cite[Theorem 1.3]{JL16}.

Let $A$ be a finite-type $W(k)$-algebra which is reduced and $\ell$-torsion free. 
Let $\pi_1$ and $\pi_2$ be co-Whittaker $A[G_n]$-modules with the same central character $\omega$. 
Let $P$ be the maximal parabolic subgroup of $G_n$ with Levi subgroup $G_{n-1} \times G_1$. 
Let $Z=Z_n$ and $U=U_n$, and $V_0$ be the canonical sub-module of $\pi_i$ which is isomorphic to c-${ \rm Ind}_{ZU}^{P} \omega\psi_A$ (see \cite[Lemma 3.1]{HM16}). We have
\begin{equation}\label{sec3equ1}
W^1_v(p)=W^2_v(p)\,,\, \forall p \in P\,,\,\forall v \in V_0\,,
\end{equation}
\begin{equation}\label{sec3equ2}
W^i_v(gp)=W^i_{\rho(p)v}(g)\,,\, \forall g \in G_n\,,\, \forall p \in P\,,\, \forall v \in V_0\,,\, i=1,2\,.
\end{equation}

We recall the decomposition of $G_n$ into double cosets of $U$ and $P$ as in \cite{Ch06}:
$$G_n = \dot\bigcup_{i=0}^{n-1}U \alpha^i P\,,\text{ where }\alpha=\begin{pmatrix}
0 & I_{n-1}\\
1 & 0
\end{pmatrix}.$$
Note that $\alpha^i=\left(\begin{smallmatrix}
0 & I_{n-i}\\
I_i & 0
\end{smallmatrix}\right)$, in particular, $\alpha^0=\alpha^n=I_n$. 

\begin{defn}\label{defn}
For each double coset $U \alpha^i P$, $0 \leq i \leq n-1$, we call $i$ the {\it height} of the double coset. We say that $\pi_1$ and $\pi_2$ {\it agree} at height $i$ if 
$$
W^1_v(g)=W^2_v(g)\,,\, \forall g \in U\alpha^i P\,,\, \forall v \in V_0\,.$$
\end{defn}

By \eqref{sec3equ1}, $\pi_1$ and $\pi_2$ agree at height 0.
The following lemma, which is the analogue of \cite[Lemma 3.1]{Ch06} with the same proof, gives a characterization of $\pi_1$ and $\pi_2$ agreeing at height $i$.

\begin{lem}\label{lem5}
$\pi_1$ and $\pi_2$ agree at height $i$
if and only if
 $$W^1_v(\alpha^i)=W^2_v(\alpha^i)\,,\, \forall v \in V_0\,.$$
\end{lem} 

\begin{defn}
For $1 \leq t \leq n-1$, we say $\pi_1$ and $\pi_2$ satisfy hypothesis $\CH_t$ if $\gamma(\pi_1\times \tau,X,\psi)=\gamma(\pi_2\times \tau,X,\psi)$ for all irreducible generic integral representations $\tau\in \Rep_{\qlb}(G_t)$. We say $\pi_1$ and $\pi_2$ satisfy hypothesis $\CH_{\leq s}$ if they satisfy $\CH_t$ for all $t\leq s$.
\end{defn}

The following lemma is the analogue of \cite[Proposition 3.1]{Ch06} with same proof. 

\begin{lem}\label{lem7}
Let $t$ with $1 \leq t \leq n-1$. If $\pi_1$ and $\pi_2$ satisfy hypothesis $\CH_t$, then
they agree at height $t$. 
\end{lem}

The following proposition is an analogue of \cite[Proposition 3.6]{JL16}, which allows us to prove Theorem \ref{main1}    
 inductively. 

\begin{prop}\label{prop3}
Assume that $\pi_1$ and $\pi_2$ satisfy hypothesis $\CH_{\leq [\frac{n}{2}]}$. Let $t$ with $[\frac{n}{2}] \leq t \leq n-2$. Suppose that for any $s$ with $0 \leq s \leq t$, the representations $\pi_1$ and $\pi_2$ agree at height $s$. Then they agree at height $t+1$.  
\end{prop}

Before proving the proposition, we apply it to the proof of our main result as follows.

\textbf{Proof of Theorem \ref{main1}}. 
Assume that $\pi_1$ and $\pi_2$ satisfy hypothesis $\CH_{\leq [\frac{n}{2}]}$. By Lemma \ref{lem7}, $\pi_1$ and $\pi_2$ agree at heights $1, 2, \ldots, [\frac{n}{2}]$. Note that by \eqref{sec3equ1}, $\pi_1$ and $\pi_2$ already agree at height 0.
Applying Proposition \ref{prop3} repeatedly for $t$ from $[\frac{n}{2}]$ to $n-2$, we obtain that $\pi_1$ and $\pi_2$ also agree at heights $[\frac{n}{2}]+1, \ldots, n-1$. 
Hence, $\pi_1$ and $\pi_2$ agree at all the heights $0, 1, \ldots, n-1$, that is, $W^1_v(g)=W^2_v(g)$, for all $g \in G_n$ and for all $v \in V_0$. Since there is some $v\in V_0$ such that $W_v^i(g)\in A^{\times}$ for some $g$, we have $\ScW(\pi_1,\psi)=\ScW(\pi_2,\psi)$ by Lemma \ref{equivalencecriterion}. This completes the proof of Theorem \ref{main1}. \qed

Therefore, we only need to prove Proposition \ref{prop3}, which will be done in Section 6.

\section{Proof of Proposition \ref{prop3}}
\label{section:proofofprop3}


First, we recall \cite[Lemma 3.5]{JL16}, which characterize certain supports of Whittaker functions $W_v^1, W_v^2$, for $v \in V_0$. 

\begin{lem}[\cite{JL16}, Lemma 3.5]\label{lem2}
Let $t$ with $[\frac{n}{2}] \leq t \leq n-2$. Suppose that for any $s$ with $0 \leq s \leq t$ the representations $\pi_1$ and $\pi_2$ agree at height $s$. Then the following equality holds
for all $X \in M_{(n-t-1) \times (2t+2-n)}(F)$, all $g \in G_{n-t-1}$, and all $v \in V_0$:
\begin{equation}
\label{lem2eqn}
W_v^1 \begin{pmatrix}
I_{n-t-1} & 0 & 0\\
0 & I_{2t+2-n} & 0\\
0 & X & g
\end{pmatrix}=W_v^2 \begin{pmatrix}
I_{n-t-1} & 0 & 0\\
0 & I_{2t+2-n} & 0\\
0 & X & g
\end{pmatrix}.
\end{equation}
\end{lem}

\textbf{Proof of Proposition \ref{prop3}}.

The proof is similar to that of \cite[Proposition 3.6]{JL16}. 

Fix any pair $(X,g)$ as in Lemma~\ref{lem2}. Then, from (\ref{lem2eqn}) we get
\begin{equation*}
W_v^1 \left(\omega_n \omega_n \begin{pmatrix}
I_{n-t-1} & 0 & 0\\
0 & I_{2t+2-n} & 0\\
0 & X & g
\end{pmatrix}\right)=W_v^2 \left(\omega_n \omega_n \begin{pmatrix}
I_{n-t-1} & 0 & 0\\
0 & I_{2t+2-n} & 0\\
0 & X & g
\end{pmatrix}\right),
\end{equation*}
that is, 
\begin{equation*}
W_v^1 \left(\omega_n \begin{pmatrix}
g_1 & X_1 & 0\\
0 & I_{2t+2-n} & 0\\
0 & 0 & I_{n-t-1}
\end{pmatrix}\omega_n\right)=W_v^2 \left(\omega_n \begin{pmatrix}
g_1 & X_1 & 0\\
0 & I_{2t+2-n} & 0\\
0 & 0 & I_{n-t-1}
\end{pmatrix}\omega_n\right),
\end{equation*}
where $g_1 = \omega_{n-t-1} g \omega_{n-t-1}$, $X_1=\omega_{n-t-1} X \omega_{2t+2-n}$. 

Note that 
$$\omega_n = 
\begin{pmatrix}
\omega_{n-t-1} & 0\\
0 & I_{t+1}
\end{pmatrix} \omega_{n,n-t-1} \alpha^{t+1}\,,\text{ with }\omega_{n,n-t-1}:=\begin{pmatrix}
I_{n-t-1} & 0\\
0 & \omega_{t+1}
\end{pmatrix}.$$
Hence,
\begin{align*}
\ & W_v^1 \left(\omega_n \begin{pmatrix}
g_2 & X_1 & 0\\
0 & I_{2t+2-n} & 0\\
0 & 0 & I_{n-t-1}
\end{pmatrix}\omega_{n,n-t-1} \alpha^{t+1}\right)\\
= \ &  W_v^2 \left(\omega_n \begin{pmatrix}
g_2 & X_1 & 0\\
0 & I_{2t+2-n} & 0\\
0 & 0 & I_{n-t-1}
\end{pmatrix}\omega_{n,n-t-1} \alpha^{t+1}\right),
\end{align*}
where $g_2 = \omega_{n-t-1} g$, $X_1=\omega_{n-t-1} X \omega_{2t+2-n}$. 

Let $X_v^i= \rho(\alpha^{t+1}) W^i_{v}$. Then
\begin{align*}
& X_v^1 \left(\omega_n \begin{pmatrix}
g_2 & X_1 & 0\\
0 & I_{2t+2-n} & 0\\
0 & 0 & I_{n-t-1}
\end{pmatrix}\omega_{n,n-t-1}\right)\\
= \ & X_v^2 \left(\omega_n \begin{pmatrix}
g_2 & X_1 & 0\\
0 & I_{2t+2-n} & 0\\
0 & 0 & I_{n-t-1}
\end{pmatrix}\omega_{n,n-t-1}\right).
\end{align*}
Recall that $\wt{X^i_v}(g)=X^i_v(\omega_n {}^t g^{-1})$. 
Then, 
\begin{align*}
& \wt{X_v^1} \left(\begin{pmatrix}
g_3 & 0 & 0\\
X_2 & I_{2t+2-n} & 0\\
0 & 0 & I_{n-t-1}
\end{pmatrix}\omega_{n,n-t-1}\right)\\
= \ & \wt{X_v^2} \left(\begin{pmatrix}
g_3 & 0 & 0\\
X_2 & I_{2t+2-n} & 0\\
0 & 0 & I_{n-t-1}
\end{pmatrix}\omega_{n,n-t-1}\right),
\end{align*}
where $g_3 = \omega_{n-t-1} {}^t g^{-1}$, $X_2=-\omega_{2t+2-n} {}^t X \omega_{n-t-1} g_1$. 

Therefore, 
\begin{align*}
& \wt{X_v^1} \left(\begin{pmatrix}
g & 0 & 0\\
X & I_{2t+2-n} & 0\\
0 & 0 & I_{n-t-1}
\end{pmatrix}\omega_{n,n-t-1}\right)\\
= \ & \wt{X_v^2} \left(\begin{pmatrix}
g & 0 & 0\\
X & I_{2t+2-n} & 0\\
0 & 0 & I_{n-t-1}
\end{pmatrix}\omega_{n,n-t-1}\right),
\end{align*}
for all $X \in M_{(2t+2-n) \times (n-t-1)}(F)$, all $g \in G_{n-t-1}$, and all $v \in V_0$. Then, by the definition of the zeta integral $\Psi$, 
we have the following equality:
\begin{align*}
& \Psi(\rho(\omega_{n,n-t-1})(\wt{X^1_v}), \wt{W_{\tau}}, \frac{q^t}{X}; 2t+2-n) \\
= \ & \Psi(\rho(\omega_{n,n-t-1})(\wt{X^2_v}), \wt{W_{\tau}}, \frac{q^t}{X}; 2t+2-n)\,,
\end{align*}
for all irreducible co-Whittaker $A'[G_{n-t-1}]$-modules $\tau$, all Whittaker functions $W_{\tau} \in \CW(\tau, \psi^{-1})$, and all $v \in V_0$. 

Since $\pi_1$ and $\pi_2$ satisfy hypothesis $\CH_{\leq [\frac{n}{2}]}$, and $n-t-1 \leq [\frac{n}{2}]$, 
by the functional equation in Theorem~\ref{fnleqn}, we have that 
$$\Psi(X^1_v, W_{\tau}, X; n-t-2) = \Psi(X^2_v, W_{\tau}, X; n-t-2)\,,$$
for all irreducible generic representations $\tau$ of $G_{n-t-1}$, all Whittaker functions $W_{\tau} \in \CW(\tau, \psi^{-1})$, and all $v \in V_0$.
Hence, by Lemma \ref{vanishinglem}, 
\begin{align*}
\int X^1_v \left(\begin{smallmatrix}
I_{n-t-1} & 0 & 0\\
X & I_{n-t-2} & 0\\
0 & 0 & I_{2t+3-n}
\end{smallmatrix}\right) dX
= \  \int X^2_v \left(\begin{smallmatrix}
I_{n-t-1} & 0 & 0\\
X & I_{n-t-2} & 0\\
0 & 0 & I_{2t+3-n}
\end{smallmatrix}\right) dX\,,
\end{align*}
for all $v \in V_0$, the integral being over $M_{(n-t-2) \times (n-t-1)}(F)$. 
We claim (Lemma \ref{lem3} below) that this identity implies in fact
$$X^1_v(I_n)=X^2_v(I_n)\,,\, \forall v \in V_0\,.$$
Taking this for granted at the moment we finish the proof. Indeed, we have then
$$W^1_v(\alpha^{t+1})=W^2_v(\alpha^{t+1})\,,\, \forall v \in V_0\,.$$
Therefore by Lemma \ref{lem5}, $\pi_1$ and $\pi_2$ agree at height $t+1$. 
This concludes the proof of Proposition \ref{prop3}. \qed

In the rest of this section we establish our claim, that is, we prove Lemma \ref{lem3} below, which is an analogue of \cite[Lemma 4.1]{JL16}. One of the basic formulas we will repeatedly use is the Fourier inversion formula. For completeness, we record the well-known fact that the Fourier inversion formula still holds in this setting.

Fix $\psi:F\to W(k)^{\times}$, let $\FY=F^k$ for some integer $k$, and let 
$$\langle,\rangle:\FY^2\to F$$
be a non-degenerate bilinear pairing on $\FY$. We define the Fourier transform 
$\widehat{\Phi}$ of $\Phi\in C_c^{\infty}(\FY,A)$ by 
$$\widehat{\Phi}(Y) := \int_{\FY}\Phi(X)\psi(\langle Y,X\rangle)dX\,,$$
 where $dX=d\mu(X)$ for a $W(k)$-valued Haar measure $\mu$ on $\FY$.

\begin{lem}[Fourier inversion formula]
\label{inversionformula}
Given $\Phi\in C_c^{\infty}(\FY,A)$, then $\widehat{\Phi}$ is also in $C_c^{\infty}(\FY,A)$ and there is a Haar measure $\mu$ on $\FY$ for which the Fourier inversion formula
$\widehat{\widehat{\Phi}}(X)=\Phi(-X)$, holds for all $\Phi\in C_c^{\infty}(\FY,A)$. If $\ell\neq 2$, this Haar measure is unique.
\end{lem}
\begin{proof}
The $A$-module $C_c^{\infty}(\FY,A)$ is spanned by the characteristic functions $\Phi_j$ of $a+\varpi^j\FY$, $a\in \FY$, $j\in \BZ$. In addition, the map $X\mapsto \psi(\langle -,X\rangle)$ gives an isomorphism from $\FY$ to the set of characters $\FY\to W(k)^{\times}$ in the usual way (cf. \cite[1.7]{BH06}). The inversion formula follows from the identity $\int_{\FY}\psi(\langle A,B\rangle)dA=0$ unless $B=0$, in which case it is a unit $u$ in $W(k)$. One can choose a Haar measure, depending on $\psi$, such that $u=1$. If $\ell\neq 2$, we may proceed exactly as in \cite[23.1]{BH06}, choosing a square root of $q$ in $W(k)^{\times}$ and letting $\mu(\ScO^k)=q^{kl/2}$,  where $l$ is the level of $\psi$.
\end{proof}

The proof of Lemma \ref{lem3} below goes exactly as that of \cite[Lemma 4.1]{JL16}, we just give the outline. 

\begin{lem}\label{lem3}
Recall that $X_v^i= \rho(\alpha^{t+1}) W^i_{v}$, $i=1,2$. If 
\begin{align}\label{lem3equ1}
\begin{split}
& \int_{M_{(n-t-2) \times (n-t-1)}(F)} X^1_v \begin{pmatrix}
I_{n-t-1} & 0 & 0\\
X & I_{n-t-2} & 0\\
0 & 0 & I_{2t+3-n}
\end{pmatrix} dX \\
= \ & \int_{M_{(n-t-2) \times (n-t-1)}(F)} X^2_v \begin{pmatrix}
I_{n-t-1} & 0 & 0\\
X & I_{n-t-2} & 0\\
0 & 0 & I_{2t+3-n}
\end{pmatrix} dX\,,
\end{split}
\end{align} 
for all $v \in V_0$, 
then $X^1_v(I_n)=X^2_v(I_n)$, for all $v \in V_0$.
\end{lem}

\begin{proof}
Since $X_v^i= \rho(\alpha^{t+1}) W^i_{v}$, by \eqref{sec3equ2}, 
equality \eqref{lem3equ1} implies that 
\begin{align}\label{lem3equ2}
\begin{split}
& \int_{M_{(n-t-2) \times (n-t-1)}(F)} X^1_v \left(u \begin{pmatrix}
I_{n-t-1} & 0 & 0\\
X & I_{n-t-2} & 0\\
0 & 0 & I_{2t+3-n}
\end{pmatrix} p\right) dX \\
= \ & \int_{M_{(n-t-2) \times (n-t-1)}(F)} X^2_v \left(u \begin{pmatrix}
I_{n-t-1} & 0 & 0\\
X & I_{n-t-2} & 0\\
0 & 0 & I_{2t+3-n}
\end{pmatrix} p\right) dX\,,
\end{split}
\end{align} 
for all $u \in U$, all $p \in \alpha^{t+1} P (\alpha^{t+1})^{-1}$, and all $v \in V_0$. 
Recall that 
$$\alpha^{t+1} = \begin{pmatrix}
0 & I_{n-t-1}\\
I_{t+1} & 0
\end{pmatrix}\,.$$
Hence the $(n-t-1)$-th row of any $p$ in $\alpha^{t+1} P (\alpha^{t+1})^{-1}$ has the form $(0, \dots, 0, a, 0, \ldots, 0)$ with $a \neq 0$ in the $(n-t-1)$-th column. 
Conversely, this condition characterizes the elements of 
$\alpha^{t+1} P (\alpha^{t+1})^{-1}$.
We will use the relation \eqref{lem3equ2} only for  $p\in U\cap \alpha^{t+1} P (\alpha^{t+1})^{-1}$. 

We recall some notation from \cite[Lemma 4.1]{JL16}. 
We denote by $\xi_{i,j}$ the matrix whose only non-zero entry 
is $1$ in the $i$-th row and $j$-th column. Thus $\xi_{i,j} \xi_{j',k}= \delta _{j,j'} \xi_{i,k} \,$.
Given a root $\alpha$ (positive or negative) we denote by $X_\alpha$ the corresponding root subgroup. Thus if $\alpha=e_i-e_j$, for any $a\in F$, the element $I_n+ a \xi_{i,j}$ is in $X_\alpha$.

Set 
\[\mathfrak{X}= \left\{\left( \begin{array}{c c c} I_{n-t-1} & 0 & 0 \\ X & I_{n-t-2} \\ 0&0& I_{2 t + 3 -n} \end{array}\right)\,,\,
X \in M_{(n-t-2) \times (n-t-1)}(F) \right\}.\]
The group $\mathfrak{X}$ is abelian and is the direct product of the groups $X_{e_a-e_b}$ with
\[ n-t\leq a \leq 2(n-t)-3\,,\, 1 \leq b\leq n -t-1 \,.\]
For such a pair $(a,b)$ we have either
\[ b\leq a-(n-t)+1 \,,\]
or
\[ a \le b + n-t-2 \,.\]

Next, we define subgroups of $\mathfrak{X}$ as in \cite[Lemma 4.1]{JL16}. 
For $n -t \leq a \leq 2(n-t)-3$, we define the following subgroup of $\mathfrak{X}$:
\[ X_a= \prod_{1 \leq b \leq a -(n-t) +1} X_{e_a-e_b}\,.\]
We also define a subgroup of $U$  as follows. 
\[  Y_a = \prod_{1 \leq b \leq a -(n-t) +1} X_{e_b -e _{a+1}} \,.\]
As discussed in \cite[Lemma 4.1]{JL16}, $Y_a$ is contained in the subgroup $U \cap \alpha^{t+1} P \alpha ^{-(t+1)}$. 
We can identify $Y_a$ with the dual of $X_a$ as follows: if for $X\in X_a$, $Y\in Y_a$, write
\begin{align*}
X =\ & I_n+ \sum _{1 \leq b \leq a -(n-t) +1} \xi_{a,b} x_b \,,\\
Y =\ & I_n + \sum_{ 1 \leq b \leq a -(n-t) +1} \xi_{b,a+1} y_b\,,
\end{align*}
then set
\[ \langle X,Y\rangle = \sum_{ 1 \leq b \leq a -(n-t) +1} x_b y_b \,.\]

For $2\leq b \leq n-t-1$, we define
\[ Z_b = \prod_{n-t\leq a \leq b+ n-t-2} X_{e_a-e_b} \,.\]
We also define a subgroup of $U$  as follows:
\[ T_b = \prod_{n-t\leq a \leq b+ n-t-2} X_{e_{b-1}-e_a} \,,\]
which is also 
contained in $U \cap \alpha^{t+1}P \alpha^{-(t+1)}$.
Again we can identify $T_b$ with the dual of $Z_b$ as follows: if for $Z\in Z_b$, $T\in T_b$, write
\begin{align*}
Z =\ & I_n+  \sum_{ n-t\leq a \leq b+ n-t-2} \xi _{a,b} z_a \,,\\
T= \ & I_n + \sum_{ n-t\leq a \leq b+ n-t-2} \xi _{b-1,a} t_a\,,
\end{align*}
then set
\[ \langle Z , T\rangle = \sum_{ n-t\leq a \leq b+ n-t-2} z_a y_a \,.\]

The group $\mathfrak{X}$ is the product
$$\prod_{n-t \leq a \leq 2(n-t)-3} X_{a} \prod_{2 \leq b \leq n-t-1} Z_b\,.$$
The identity \eqref{lem3equ1} can be written as follows: for all $v \in V_0$, 
\[\int_{\mathfrak{X}} X_v^1(X) d X = \int_{\mathfrak{X}} X_v^2 (X ) d X \,.\] 
Note that the two functions $X^i_v$ on $\mathfrak{X}$ are smooth and compactly supported. We should keep in mind that
\[ X_{\rho(p)v}^i (X) = X_v^i( X (\alpha^{t+1} p {\alpha^{-(t+1)}} )) \,,\, \forall p \in P\,,\, \forall v \in V_0\,.\]

We now list the three main steps, referring to the proof of \cite[Lemma 4.1]{JL16} for details. For each step, the coefficient ring $A$ is involved in the argument only when the Fourier inversion formula is used at the conclusion of each step.

\textbf{First step.} We show that we have, for all $v \in V_0$, the identity
\[ \int X_v^1(X) d X = \int X_v^2 (X ) d X \, ,\]
where both integrals are over the product
$$\prod_{n-t \leq a \leq 2(n-t)-4} X_{a} \prod_{2 \leq b \leq n-t-1} Z_b\,.$$

By \eqref{lem3equ2}, for all $Y \in Y_{2(n-t)-3}=\prod _{1\leq b \leq n-t-2} X_{e_b -e _{2(n-t)-2}}$ and all $v \in V_0$, we know that
 \[\int X_v^1(XY ) d X = \int X_v^2 (X Y ) d X \,,\] 
where both integrals are over the product
$$\prod_{n-t \leq a \leq 2(n-t)-3} X_{a} \prod_{2 \leq b \leq n-t-1} Z_b\,.$$  

As in the proof of \cite[Lemma 4.1]{JL16}, we get, for all $Y \in Y_{2(n-t)-3}$ and all $v \in V_0$,
\[ \int X_v^1 (X)  \psi_A(\langle X^{2(n-t)-3}, Y\rangle) d X =\int X_v^2 (X) \psi_A(\langle X^{2(n-t)-3}, Y\rangle)  d X \,,\] where $ X^{2(n-t)-3}$ is the projection of $X$ on the subgroup $X_{2(n-t)-3}$. and both integrals are over the product
$$\prod_{n-t \leq a \leq 2(n-t)-3} X_{a} \prod_{2 \leq b \leq n-t-1} Z_b\,.$$
Now we apply the Fourier inversion formula of Lemma~\ref{inversionformula} to the group $\FY=X_{2(n-t)-3}$ and the Schwartz function $\Phi$ on $\FY$ given by 
$$\int(X_v^1(X) - X_v^2(X))dX\,,$$
 where the integral runs over the product $$\prod_{n-t \leq a \leq 2(n-t)-4} X_{a} \prod_{2 \leq b \leq n-t-1} Z_b\,,$$ and the $X_{2(n-t)-3}$ factor of $X$ is the variable of $\Phi$. Then $\widehat{\Phi}$ is identically $0$, so Lemma~\ref{inversionformula} implies $\Phi=0$, completing the First step.

\textbf{Second step.}
 Assume that for $k$ with $n-t\leq k\leq 2(n-t)-4$ and for all $v \in V_0$, we have established the identity 
 \[ \int X^1_v(X) d X = \int X^2_v(X) d X\,,\]
where both integrals are over the product
\[ \prod_{n-t \leq a \leq k\,,\, 2 \leq b \leq n-t-1} X_a Z_ b\,.\]
We show that for all $v \in V_0$, we have the identity
 \[ \int X^1_v(X) d X = \int X^2_v(X) d X\,,\]
where both integrals are over the product
\[ \prod_{n-t \leq a \leq k-1\,,\, 2 \leq b \leq n-t-1} X_a Z_ b\,.\] The Fourier inversion formula is used in a similar way to the First step.

\textbf{Third step.} Applying descending induction on $k$ we arrive at
\[ \int_{\prod_{2 \leq b \leq n-t-1}  Z_b} X_v^1 (Z) d Z = \int_{\prod_{2 \leq b \leq n-t-1}  Z_b} X_v^2 (Z) d Z\,,\, \forall v \in V_0\,. \]
We prove now that for $2\leq k \leq n-t-1$, if we have
\[ \int_{\prod_{ k \leq b \leq n-t-1}Z_b} X_v^1 (Z) d Z = \int_{\prod_{ k \leq b \leq n-t-1}Z_b} X_v^2 (Z) d Z\,,\, \forall v \in V_0\,,\]
then we have 
\[ \int_{\prod_{k+1 \leq b \leq n-t -1} Z_b} X_v^1 (Z) d Z = \int_{\prod_{k+1 \leq b \leq n-t -1} Z_b} X_v^2 (Z) d Z\,,\, \forall v \in V_0\,. \] The Fourier inversion formula is used in a similar way to the First step.
By ascending induction this will complete the proof of the lemma. 
\end{proof}


\section{A Descent Theorem}
\label{section:descent}

Let $(\pi, V)$ be a co-Whittaker $A[G_n]$-module. Let $A'$ be a sub-$W(k)$-algebra of $A$.

\begin{defn}
\begin{enumerate}
\item $\pi$ satisfies hypothesis $\CH_0(A')$ if its central character 
$\omega_{\pi}:Z\to A^{\times}$ factors through the inclusion $(A')^{\times}\subset A^{\times}$.
\item For $t\geq 1$, we say that $\pi$ satisfies hypothesis $\CH_t(A')$ if $\gamma(\pi \times e'W_t,X,\psi)$ has coefficients in $A'\otimes e'\ScZ_t$ for all primitive idempotents $e'$ of $\ScZ_t$. It satisfies $\CH_{\leq s}(A')$ if it satisfies $\CH_t(A')$ for all $t\leq s$.
\end{enumerate}
\end{defn}

In this section we prove the following version of Theorem~\ref{descent:intro}.

\begin{thm}
\label{descent}
Assume that $A$ is a finite extension of $A'$ and $\pi$ satisfies hypothesis $\CH_{\leq [n/2]}(A')$. Then the supercuspidal support map 
$$f_V: \ScZ_n \rightarrow A$$
 factors through the inclusion $A' \subset A$.
\end{thm}

In \cite{HM16}, following version of descent theorem has been proved. 

\begin{thm}[\cite{HM16}, Theorem 3.2]
\label{heightsdescent}
Assume that $A$ is a finite extension of $A'$ and $\pi$ satisfies hypothesis $\CH_{\leq n-1}(A')$. 
Then $f_V$ factors through the inclusion $A' \subset A$.
\end{thm}

Recall that $P$ is the maximal parabolic subgroup of $G_n$ with Levi subgroup $G_{n-1} \times G_1$, $Z=Z_n$ is the center of $G_n$ and $U=U_n$. 
Also recall that $V_0$ is the canonical sub-module of $V$ which is isomorphic to 
$\cInd_{ZU}^{P}(\omega_{\pi}\psi_A)$.
Let $V_0'$ be the canonical sub-$W(k)$-module of $V_0$ which is isomorphic to  
 $\cInd_{ZU}^{P}(\omega_{\pi}\psi_{A'})$. 
We recall the following decomposition of $G_n$ from Section \ref{section:proofmain1}: 
$$G_n = \dot\bigcup_{i=0}^{n-1}U \alpha^i P\,,\text{ where } 
\alpha=\begin{pmatrix}
0 & I_{n-1}\\
1 & 0
\end{pmatrix}.$$

\begin{defn}
We say $\pi$ is $A'$-valued at height $t$ if $W_v(g)$ lies in $A'$, for all $g \in U\alpha^tP$ and all $v \in V_0'$. Note that if $\CH_0(A')$ is satisfied, $\pi$ is $A'$-valued at height $t$ if and only if $W_v(\alpha^i) \in A'$, for all $v \in V_0'$.
\end{defn}

Note that if $\CH_0(A')$ is satisfied, $\pi$ is $A'$-valued at height 0.

\begin{prop}\label{centralcharactervaluedA'}
Suppose $n\geq 2$ and $\pi$ satisfies $\CH_1(A')$. Then $\pi$ satisfies $\CH_0(A')$.
\end{prop}
\begin{proof}
If $\chi:F^{\times} \to W(k)^{\times}$ is any smooth character, then it is trivially co-Whittaker and has a supercuspidal support map $f_{\chi}:\ScZ_1 \to W(k)$. By compatibility of the gamma factor with base change, we thus have
$$(id\otimes f_{\chi})(\gamma(\pi\times W_1,X,\psi))=\gamma(\pi\times \chi,X,\psi)$$ has coefficients in $A'\otimes_{W(k)}W(k)$ by hypothesis $\CH_1$. Taking $\chi$ to have sufficiently large conductor $m$, Proposition \ref{prop:central} allows us to conclude $\omega_{\pi}(c)^{-1}\gamma(1_A\times\chi,X,\psi)^n$ has coefficients in $A'\otimes W(k)$ for any $c\in \Fp^{-m}$ satisfying $\chi(1+x)=\psi(cx)$ for $x\in \Fp^{[m/2]+1}$. But since the supercuspidal support of $1_A$ factors through $A'$ (in fact, $W(k)$), $\gamma(1_A\times\chi,X,\psi)$ is valued in $S^{-1}(A'\otimes W(k)[X,X^{-1}])$. Since $\gamma(1_A\times\chi,X,\psi)$ is a unit in this ring (\cite[Cor 5.6]{M16}) we have $\omega_{\pi}(c)^{-1}\otimes 1$ is in $A'\otimes W(k)$, showing $\omega_{\pi}(c)^{-1}$ is in $A'$. For any $c\in \Fp^{-m}\backslash \Fp^{1-m}$, there exists $\chi_c$ of conductor $m$ such that $\chi_c(1+x)=\psi(cx)$ for $x\in\Fp^{[m/2]+1}$. Since any element of $F^{\times}$ can be expressed as the quotient of two elements of valuation at most $-m$, we deduce that $\omega_{\pi}$ is valued in $A'$.
\end{proof}

The functional equation requires that the zeta integrals lie in the ring $S^{-1}R[X,X^{-1}]$ of ``rational functions'', however, we only know this rationality is preserved by \emph{finite} descent, in the following sense.

\begin{prop}[\cite{HM16}, Corollary 4.2]
\label{rationalityfinite}
Suppose $R'$ is a Noetherian $W(k)$-subalgebra of $R$ such that $R$ is finitely generated as an $R'$-module. Let $S'$ be the subset of $R'[X,X^{-1}]$ consisting of polynomials whose first and last nonzero coefficients are units in $R'$. Then $(S')^{-1}R'[X,X^{-1}]$ is the intersection, in $R[[X]][X^{-1}]$, of the subrings $R'[[X]][X^{-1}]$ and $S^{-1}R[X,X^{-1}]$.
\end{prop}

In what follows, we will refer to Proposition~\ref{rationalityfinite} when $R=A\otimes e'\ScZ_t$ and $R'$ is the subring $A'\otimes e'\ScZ_{t}$, for some $e'$. Note that this is indeed a subring since $e'\ScZ_t$ is flat over $W(k)$.

The following proposition is the analogue of \cite[Proposition 3.1]{Ch06}.

\begin{prop}
\label{chen3.1analogue}
Suppose $A$ is a finite extension of $A'$, and suppose $\pi$ satisfies $\CH_0(A')$. Let $t$ be an integer with $1\leq t \leq n-1$. If $\pi$ satisfies hypothesis $\CH_t(A')$, then $\pi$ is $A'$-valued at height $t$.
\end{prop}
\begin{proof}
Since $\pi$ satisfies $\CH_0$, by definition of $V_0'$, $W_v(p) \in A'$, $\forall p \in P$, $\forall v \in V_0'$. Hence, 
$$W_v \begin{pmatrix}
g_t &&\\
x &I_{n-t-1}&\\
&&1
\end{pmatrix} \in A'\,, \,\forall v \in V_0'\,,$$
where $g_t$ is any element in $G_t$, $x \in M_{(n-t-1) \times t}$. Then, for any primitive idempotents $e'$ of $\ScZ_t$, 
for any $W' \in \CW(e'W_t, \overline{\psi})$, 
\begin{align*}
& \ \Psi(W,W',X;n-t-1)\\
= & \ \sum_{r\in \BZ}\int_{M_{n-t-1,t}(F)}\int_{N_t\backslash\{g\in G_t:v(\det g)=r\}}                    \left(W\left(\begin{smallmatrix}g &&\\x&I_{n-t-1} &\\&&1 \end{smallmatrix}\right)\otimes W'(g)\right)X^rdg dx\\
\in & \  (A' \otimes e' \ScZ_t)[[X]][X^{-1}]\,.
\end{align*}

Since $A'\subset A$ is finite, by Proposition \ref{rationalityfinite}, $\Psi(W,W',X;n-t-1)$ lies in $(S')^{-1}(A'\otimes e'\ScZ_t)[X,X^{-1}]$. Applying the involution $X \mapsto \frac{q^{n-t-1}}{X}$ gives 
$$\Psi(W,W',\frac{q^{n-t-1}}{X};n-t-1)\in (S')^{-1}(A'\otimes e'\ScZ_t)[X,X^{-1}]\,.$$

By assumption, $\pi$ satisfies hypothesis $\CH_t$, that is, 
$\gamma(\pi \times e'W_t,\frac{q^{n-t-1}}{X},\psi)$ has coefficients in $A'\otimes e'\ScZ_t$ for all primitive idempotents $e'$ of $\ScZ_t$. 
Hence, 
$$\Psi(\omega_{n,t}\wt{W}, \wt{W'}, X; 0) \in (A' \otimes e' \ScZ_t)[[X]][X^{-1}]\,.$$

Therefore, 
$$\wt{W} \left(\begin{pmatrix}
g_t &\\
& I_{n-t}
\end{pmatrix} \begin{pmatrix}
I_t&\\
&\omega_{n-t} 
\end{pmatrix} \right) \in A' \otimes e' \ScZ_t\,, \,\forall g_t \in G_t\,, \,\forall v \in V_0'\,.$$

Take $g_t=I_t$, we get that 
$$\wt{W}  \begin{pmatrix}
I_t&\\
&\omega_{n-t} 
\end{pmatrix} \in A' \otimes e' \ScZ_t\,, \,\forall v \in V_0'\,,$$
that is, 
$${W}  \left( \omega_n \begin{pmatrix}
I_t&\\
&\omega_{n-t} 
\end{pmatrix} \right) \in A' \otimes e' \ScZ_t\,, \,\forall v \in V_0'\,,$$
which is exactly
$${W}  (\alpha^t) \in A' \otimes e' \ScZ_t\,, \,\forall v \in V_0'\,.
$$
Therefore, $\pi$ is $A'$-valued at height $t$. This completes the proof of the proposition. 
\end{proof}

The following proposition is an analogue of \cite[Proposition 3.6]{JL16}.
\begin{prop}
\label{JL3.6analogue}
Assume $A$ is a finite extension of $A'$ and $\pi$ satisfies hypothesis $\CH_{\leq [n/2]}(A')$. Let $t$ be such that $[n/2]\leq t \leq n-2$. Suppose that for any $s$ with $0\leq s \leq t$, $\pi$ is $A'$-valued at height $s$. Then $\pi$ is $A'$-valued at height $t+1$.
\end{prop}
Before proving Proposition~\ref{JL3.6analogue}, we use it to deduce Theorem \ref{descent}.
\begin{proof}[Proof of Theorem \ref{descent}]
Assume $\pi$ satisfies hypothesis $\CH_{\leq[n/2]}$. By Proposition \ref{chen3.1analogue}, $\pi$ is $A'$-valued at heights $1,2,\dots,[n/2]$. Note that $\pi$ is already $A'$-valued at height 0 because it satisfies $\CH_1$. Applying Proposition \ref{JL3.6analogue} repeatedly for $t$ ranging from $[n/2]$ to $n-2$, we find that $\pi$ is $A'$-valued at heights $[n/2]+1,\dots,n-1$. Hence $\pi$ is $A'$-valued at all heights $0,1,\dots,n-1$, and Lemma \ref{heightsdescent} finishes the proof.
\end{proof}

We now prove Proposition~\ref{JL3.6analogue}.

\begin{proof}[Proof of Proposition~\ref{JL3.6analogue}]
It is shown in the course of proving \cite[Lemma 3.5]{JL16} that, for all $X \in M_{(n-t-1) \times (2t+2-n)}(F)$, all $g \in G_{n-t-1}$, and all $v \in V_0'$, the matrices $\begin{pmatrix}
I_{n-t-1} & 0 & 0\\
0 & I_{2t+2-n} & 0\\
0 & X & g
\end{pmatrix}$ are in $U \alpha^{n-i} P$ with $n-i \leq t$. Therefore, by assumption, 
\begin{equation*}
W_v \begin{pmatrix}
I_{n-t-1} & 0 & 0\\
0 & I_{2t+2-n} & 0\\
0 & X & g
\end{pmatrix}\in A'\,.
\end{equation*}
Then, 
\begin{equation*}
W_v \left(\omega_n \omega_n \begin{pmatrix}
I_{n-t-1} & 0 & 0\\
0 & I_{2t+2-n} & 0\\
0 & X & g
\end{pmatrix}\right)\in A'\,,
\end{equation*}
that is, 
\begin{equation*}
W_v \left(\omega_n \begin{pmatrix}
g_1 & X_1 & 0\\
0 & I_{2t+2-n} & 0\\
0 & 0 & I_{n-t-1}
\end{pmatrix}\omega_n\right)\in A'\,,
\end{equation*}
where $g_1 = \omega_{n-t-1} g \omega_{n-t-1}$, $X_1=\omega_{n-t-1} X \omega_{2t+2-n}$. 

Note that 
$$\omega_n = 
\begin{pmatrix}
\omega_{n-t-1} & 0\\
0 & I_{t+1}
\end{pmatrix} \omega_{n,n-t-1} \alpha^{t+1}\,.$$
Recall that 
$$\omega_{n,n-t-1}=\begin{pmatrix}
I_{n-t-1} & 0\\
0 & \omega_{t+1}
\end{pmatrix}\,.$$
Hence,
\begin{align*}
\ & W_v \left(\omega_n \begin{pmatrix}
g_2 & X_1 & 0\\
0 & I_{2t+2-n} & 0\\
0 & 0 & I_{n-t-1}
\end{pmatrix}\omega_{n,n-t-1} \alpha^{t+1}\right)\in A'\,,
\end{align*}
where $g_2 = \omega_{n-t-1} g$, $X_1=\omega_{n-t-1} X \omega_{2t+2-n}$. 

Let $X_v= \rho(\alpha^{t+1}) W_{v}$. Then
\begin{align*}
& X_v \left(\omega_n \begin{pmatrix}
g_2 & X_1 & 0\\
0 & I_{2t+2-n} & 0\\
0 & 0 & I_{n-t-1}
\end{pmatrix}\omega_{n,n-t-1}\right)
\in A'\,.
\end{align*}
Recall that $\wt{X_v}(g)=X_v(\omega_n {}^t g^{-1})$. 
Then, 
\begin{align*}
& \wt{X_v} \left(\begin{pmatrix}
g_3 & 0 & 0\\
X_2 & I_{2t+2-n} & 0\\
0 & 0 & I_{n-t-1}
\end{pmatrix}\omega_{n,n-t-1}\right)
\in A'\,,
\end{align*}
where $g_3 = \omega_{n-t-1} {}^t g^{-1}$, $X_2=-\omega_{2t+2-n} {}^t X \omega_{n-t-1} g_1$. 

Therefore, 
\begin{align*}
& \wt{X_v} \left(\begin{pmatrix}
g & 0 & 0\\
X & I_{2t+2-n} & 0\\
0 & 0 & I_{n-t-1}
\end{pmatrix}\omega_{n,n-t-1}\right)
\in A'\,,
\end{align*}
for all $X \in M_{(2t+2-n) \times (n-t-1)}(F)$, all $g \in G_{n-t-1}$, and all $v \in V_0'$. Then, by the definition of the zeta integral $\Psi$, and since $A$ is finite over $A'$ we have (Proposition~\ref{rationalityfinite}):
\begin{align*}
\Psi(\rho(\omega_{n,n-t-1})(\wt{X_v}), \wt{W'}, \frac{q^t}{X}; 2t+2-n) 
\text{ is in }(S')^{-1}R'[X,X^{-1}]\,,
\end{align*}
for all idempotents $e'$ of $\ScZ_{n-t-1}$ and all $W' \in \CW(e'W_{n-t-1}, \psi^{-1})$, and all $v \in V_0'$.

Since $\pi$ satisfies hypothesis $\CH_{\leq [\frac{n}{2}]}$, and $n-t-1 \leq [\frac{n}{2}]$, the functional equation (Theorem~\ref{fnleqn}) gives 
$$\Psi(X_v, W', X; n-t-2) \in R'[[X]][X^{-1}]\,,$$
for all irreducible generic representations $\tau$ of $G_{n-t-1}$, all Whittaker functions $W' \in \CW(e'W_{n-1}, \psi^{-1})$, and all $v \in V_0'$.
Hence, by Lemma~\ref{vanishinglemdescent}
\begin{align*}
& \int_{M_{(n-t-2) \times (n-t-1)}(F)} X_v \begin{pmatrix}
I_{n-t-1} & 0 & 0\\
X & I_{n-t-2} & 0\\
0 & 0 & I_{2t+3-n}
\end{pmatrix} dX 
\in A'\,,
\end{align*}
for all $v \in V_0'$. 
We claim (Lemma \ref{lem3descent} below) that this identity implies in fact
$$X_v(I_n)\in A'\,,\, \forall v \in V_0'\,,$$ showing that
$$W_v(\alpha^{t+1})\in A'\,,\, \forall v \in V_0'\,,$$ i.e. $\pi$ is $A'$-valued at height $t+1$. This concludes the proof of Proposition~\ref{JL3.6analogue}.
\end{proof}

In the rest of this section we establish our claim, that is, we prove Lemma \ref{lem3descent} below, which is an analogue of  \cite[Lemma 4.1]{JL16}, with similar proofs. One of the basic formulas we will repeatedly use is following Fourier descent formula.

\begin{lem}(Fourier descent formula)
\label{fourierinversiondescent}
Suppose $\FY=F^k$ for some integer $k$, and $\Phi\in C_c^{\infty}(\FY,A)$. If $\widehat{\Phi}(X)\in A'$ for all $X\in \FY$, then $\Phi(X)$ is in $A'$ for all $X\in \FY$.
\end{lem}
\begin{proof}
This follows immediately from Lemma \ref{inversionformula} after noting that $\widehat{\widehat{\Phi}}$ is also valued in $A'$, since the integral is a finite sum.
\end{proof}

\begin{lem}
\label{lem3descent}
Let $X_v:=\rho(\alpha^{t+1})W_v$. If $$\label{lem3equ1descent}\int_{M_{(n-t-2) \times (n-t-1)}(F)} X_v \begin{pmatrix}
I_{n-t-1} & 0 & 0\\
X & I_{n-t-2} & 0\\
0 & 0 & I_{2t+3-n}
\end{pmatrix} dX$$ is in $A'$ for all $v \in V_0'$, then $X_v(I_n)$ is in $A'$ for all $v \in V_0'$.
\end{lem}

\begin{proof}
Since $X_v^i= \rho(\alpha^{t+1}) W^i_{v}$, and $\CH_0$ is satisfied, 
\begin{align}\label{lem3equ2descent}
\int_{M_{(n-t-2) \times (n-t-1)}(F)} X_v \left(u \begin{pmatrix}
I_{n-t-1} & 0 & 0\\
X & I_{n-t-2} & 0\\
0 & 0 & I_{2t+3-n}
\end{pmatrix} p\right) dX \in A'\,,
\end{align} 
for all $u \in U$, all $p \in \alpha^{t+1} P (\alpha^{t+1})^{-1}$, and all $v \in V_0'$. 

We keep using notation in Lemma~\ref{lem3} and \cite[Lemma 4.1]{JL16}. 

The hypothesis of Lemma~\ref{lem3descent} can be written as follows: for all $v \in V_0'$, 
\[\int_{\mathfrak{X}} X_v(X) d X \in A' \,.\] 
Note that the function $X_v$ on $\mathfrak{X}$ is smooth and compactly supported. 

As in the proof of Lemma~\ref{lem3}, we now list the three main steps. For details, we refer to the proof of \cite[Lemma 4.1]{JL16}. We only give the details of the Fourier inversion argument for the First step.

\textbf{First step.} We show that we have, for all $v \in V_0'$, 
\[ \int X_v(X) d X \in A' \, ,\]
where the integral is over the product
$$\prod_{n-t \leq a \leq 2(n-t)-4} X_{a} \prod_{2 \leq b \leq n-t-1} Z_b\,.$$

By \eqref{lem3equ2descent}, for all $Y \in Y_{2(n-t)-3}=\prod _{1\leq b \leq n-t-2} X_{e_b -e _{2(n-t)-2}}$ and all $v \in V_0'$, we know that
 \[\int X_v(XY ) d X \in A' \,,\] 
where the integral is over the product
$$\prod_{n-t \leq a \leq 2(n-t)-3} X_{a} \prod_{2 \leq b \leq n-t-1} Z_b\,.$$  

As in the proof of \cite[Lemma 4.1]{JL16}, we get, for all $Y \in Y_{2(n-t)-3}$ and all $v \in V_0'$, that this integral is the same as
\[ \int X_v (X)  \psi(\langle X^{2(n-t)-3}, Y\rangle) d X\in A' \,,\] where $ X^{2(n-t)-3}$ is the projection of $X$ on the subgroup $X_{2(n-t)-3}$ and the integral is over the product
$$\prod_{n-t \leq a \leq 2(n-t)-3} X_{a} \prod_{2 \leq b \leq n-t-1} Z_b\,.$$
Now we apply Lemma~\ref{fourierinversiondescent} to the group $\FY=X_{2(n-t)-3}$ and the Schwartz function $\Phi$ on $\FY$ given by $\Phi(Y)=\int X_v(XY)dX$, where $Y \in \FY$ and the integral runs over the product 
$$\prod_{n-t \leq a \leq 2(n-t)-4} X_{a} \prod_{2 \leq b \leq n-t-1} Z_b\,.$$
Then $\widehat{\Phi}$ takes values in $A'$, so Lemma~\ref{fourierinversiondescent} implies $\Phi$ takes values in $A'$, completing the First step.

\textbf{Second step.}
 Assume that for $k$ with $n-t\leq k\leq 2(n-t)-4$ and for all $v \in V_0'$, we have established the identity 
 \[ \int X^1_v(X) d X = \int X^2_v(X) d X\,,\]
where both integrals are over the product
\[ \prod_{n-t \leq a \leq k\,,\, 2 \leq b \leq n-t-1} X_a Z_ b\,.\]
We show that for all $v \in V_0$, we have the identity
 \[ \int X^1_v(X) d X = \int X^2_v(X) d X\,,\]
where both integrals are over the product
\[ \prod_{n-t \leq a \leq k-1\,,\, 2 \leq b \leq n-t-1} X_a Z_ b\,.\] The Fourier descent formula is used in a similar way to the First step.

\textbf{Third step.} Applying descending induction on $k$ we arrive at
\[ \int_{\prod_{2 \leq b \leq n-t-1}  Z_b} X_v^1 (Z) d Z = \int_{\prod_{2 \leq b \leq n-t-1}  Z_b} X_v^2 (Z) d Z\,,\, \forall v \in V_0'\,. \]
We prove now that for $2\leq k \leq n-t-1$, if we have
\[ \int_{\prod_{ k \leq b \leq n-t-1}Z_b} X_v^1 (Z) d Z = \int_{\prod_{ k \leq b \leq n-t-1}Z_b} X_v^2 (Z) d Z\,,\, \forall v \in V_0'\,,\]
then we have 
\[ \int_{\prod_{k+1 \leq b \leq n-t -1} Z_b} X_v^1 (Z) d Z = \int_{\prod_{k+1 \leq b \leq n-t -1} Z_b} X_v^2 (Z) d Z\,,\, \forall v \in V_0'\,. \] The Fourier descent formula is used in a similar way to the First step.
By ascending induction this will complete the proof of the lemma. 
\end{proof}

\end{document}